\documentclass{AIMS}
\usepackage{amsmath}
\usepackage{paralist}
\usepackage{graphics} 
\usepackage{epsfig} 
\usepackage{graphicx}  \usepackage{epstopdf}
\usepackage[colorlinks=true]{hyperref}
\usepackage{enumerate}
\hypersetup{urlcolor=blue, citecolor=red}

\numberwithin{equation}{section}
\theoremstyle{plain}

\newtheorem{lem}{Lemma}[section]
\newtheorem{condition}{Condition}
\newtheorem{thm}[lem]{Theorem}\newtheorem{assum}[condition]{Assumptions}

{\theoremstyle{definition}
\newtheorem{Def}[lem]{Definition}
\newtheorem{Rem}[lem]{Remark}}

\newenvironment{mlist}
{\begin{list}{ $\circ$}
		{
			\setlength{\leftmargin}{0.3cm}\setlength{\partopsep}{-0.1cm}
			\setlength{\topsep}{0.0cm}
			\setlength{\itemsep}{-0.1cm}
	}}{\end{list}}

\newcommand\dela[1]{}

\newcommand{\h}{\mathrm{H}}
\newcommand{\bH}{\mathbb{H}}
\newcommand{\sH}{\mathbb{H}_{\mathrm{sol}}}

 \usepackage{mathrsfs}
 \usepackage{bbm, dsfont}
 \usepackage{amsmath,amsthm,amsfonts,amssymb, mathabx}
 \usepackage{amsmath, enumerate}
 \usepackage{multirow}

\usepackage{soul}
\newcommand\toup{\nearrow}

\renewcommand{\v}{\bv}

\newcommand{\A}{\mathbf{A}}
\newcommand{\f}{\mathbf{F}}

\newcommand{\bw}{\mathbb{W}}
\newcommand{\y}{\mathbf{y}}
\renewcommand{\d}{\mathbf{d}}
\newcommand{\bh}{\mathbf{h}}
\newcommand{\bk}{\mathbf{k}}

\newcommand{\MO}{\mathscr{O}}

\newcommand{\bv}{\mathbf{v}}

\newcommand{\bd}{\mathbf{d}}

\newcommand{\el}{\mathbb{L}}

\newcommand{\mo}{\mathscr{O}}

\newcommand{\bu}{\mathbf{u}}

\newcommand{\eps}{\varepsilon}
\newcommand{\rve}{\rVert}
\newcommand{\lve}{\lVert}

\newcommand{\bm}{\mathbf{m}}

\newcommand{\rA}{\mathrm{A}}

\newcommand{\rrA}{\mathrm{A}_2}

\newcommand{\rH}{\mathrm{ H}}
\newcommand{\rK}{\mathrm{K}}
\newcommand{\rV}{\mathrm{ V}}
\newcommand{\divv}{\mathrm{div }\;}

\newcommand{\err}{\mathbb{R}}

\newcommand{\sol}{\text{sol}}

\title[Stochastic Ericksen-Leslie equations]{A note on the stochastic Ericksen-Leslie equations for nematic liquid crystals}

\author[Z. Brze\'zniak, E. Hausenblas and P. A. Razafimandimby]{}


 \email{zdzislaw.brzezniak@york.ac.uk}
\email{erika.hausenblas@unileoben.ac.at}
 \email{paul.razafimandimby@up.ac.za}

\thanks{E. Hausenblas is supported by the FWF-Austrian Science Fund through the Stand-Alone grant number P28010}

\thanks{This article is part of a project that is currently funded by the  European Union's Horizon 2020 research and innovation programme under the Marie Sk\l{}odowska-Curie grant agreement No. 791735 ``SELEs". Part of this work was written while P. A. Razafimandimby was at the University of Pretoria; he is grateful to the funding he received from the National Research Foundation South Africa (Grant Numbers 109355 and 112084). He is also grateful to the European Mathematical Society for the EMS-Simons for Africa-Collaborative research grant which enables him to visit Montanuniversit\"at Leoben, Austria.}

\thanks{$^*$ Corresponding author: Paul Razafimandimby}
\begin{document}
\maketitle

\centerline{\scshape Zdzis\l aw Brze\'zniak }
\medskip
{\footnotesize
 \centerline{Department of Mathematics}
   \centerline{University of York}
   \centerline{ Heslington Road, York YO10 5DD, UK}
} 

\medskip

\centerline{\scshape Erika Hausenblas}
\medskip
{\footnotesize
 \centerline{ Department of Mathematics and Information Technology}
   \centerline{ Montanuniversit\"at Leoben}
   \centerline{Franz Josef Stra\ss e 18, 8700 Leoben, Austria}
}
\medskip
\centerline{\scshape Paul Andr\'e Razafimandimby}
\medskip
{\footnotesize
 \centerline{ Department of Mathematics and Applied Mathematics}
   \centerline{University of Pretoria}
   \centerline{Lynwood Road, Pretoria 0002, South Africa}

 \centerline{Current Address:} {\footnotesize
 	\centerline{Department of Mathematics}
 	\centerline{University of York}
 	\centerline{ Heslington Road, York YO10 5DD, UK}
 }
}

\bigskip

\centerline{(Communicated by Bj\"orn Schmalfu\ss)}

	
\begin{abstract}
	In this note we prove the existence and uniqueness of local maximal smooth solution of the stochastic simplified Ericksen-Leslie systems modelling the dynamics of nematic liquid crystals under stochastic perturbations.
\end{abstract}

	\section{Introduction }\label{sec-Intro}

Liquid crystal, which is a state of matter that has properties between
amorphous liquid and crystalline solid, can be classified into two groups according to the form of their molecules. Liquid crystals with rod-shaped molecules are called calamitics while those with disc-like molecules are referred to discotics. In its turn, the calamitics can be divided into two phases: nematic and smectic. The nematic phase, referred to as nematic liquid crystal, is the simplest of liquid crystal phases.   Nematic liquid crystals
tend to align along a particular direction denoted by a unit vector $\mathbf{d}$, called the optical director axis.  Most of the interesting
phenomenology of nematic liquid crystals are linked to the geometry and
dynamics of this director. We refer to \cite{Chandrasekhar}
and \cite{Gennes} for a comprehensive treatment of the physics of
liquid crystals.

To model the hydrodynamics of nematic liquid crystals, most scientists
use the continuum theory developed by Ericksen \cite{Ericksen} and
Leslie \cite{Leslie}. From this theory, F. Lin and C. Liu
\cite{Lin-Liu} derived the most basic and simplest form of the dynamical system describing the motion of nematic liquid crystals
flowing in $\mathbb{R}^d (d=2,3)$. This
system is given by
\begin{subequations}\label{erick-leslie}
	\begin{align}
		& \bv_t+(\bv\cdot\nabla)\bv-\Delta \bv+\nabla p=-\lambda \nabla\cdot(\nabla \bd \odot \nabla \bd),\label{erick-leslie1}\\
		& \nabla\cdot \bv=0,\\
		& \bd_t+(\bv\cdot\nabla)\bd=\gamma(\Delta \bd+|\nabla \bd|^2\bd),\\
		& |\bd|^2=1.\label{erick-leslie4}
	\end{align}
\end{subequations}

Here $p:\mathbb{R}^d\to \mathbb{R}$, $\bv: \mathbb{R}^d \to \mathbb{R}^d$  and $\bd: \mathbb{R}^d\to \mathbb{R}^3$ represent the pressure, velocity of the fluid and the optical director, respectively. The symbol $\nabla \bd \odot
\nabla \bd$ stands for   a square $d\times d$-matrix with entries
given by
\begin{equation*}
	[\nabla \bd \odot \nabla \bd]_{i,j}=\sum_{k=1}^3 \frac{\partial
		\bd^k}{\partial x_i}\frac{\partial \bd^k}{\partial x_j},\;\;
	\mbox{ for any } i,j=1,\dots, d.
\end{equation*}
Since the work of  Lin and Liu \cite{Lin-Liu}, the Ericksen-Leslie system \eqref{erick-leslie}, its Ginzburg-Landau approximation,
in which the term $|\nabla \bd|^2\bd$ is replaced by $-\frac1\eps^2(\lvert \bd\rvert^2-1)\bd$ where $\eps>0$,
and their several generalizations, have been the subjects of intensive mathematical studies. We refer, among others, to \cite{Prohl,Cavaterra,Rojas-Medar2,Grasselli,Hong,Lin-Liu,
    Lin-Liu2,Lin-Wang,Lin-Lin-Wang,Liu-Walkington1,WZZ} for results obtained prior to 2013, and to \cite{Dai+Schonbeck_2014,Gal16,Hieber-16,Hong14,Huang14,Huang16,Lin+Wang16,Wang2014,Wang2016,WZZ15} for results obtained after 2014. For a detailed review of the literature about the mathematical theory of nematic liquid crystals and other related models, we recommend the review articles \cite{Lin-Wang14,Climent} and the recent paper \cite{Hieber-17}.


In this paper  we consider the following system of stochastic partial differential equations (SPDEs)
\begin{subequations}\label{F-LC}
	\begin{align}
		&	d\bv+\biggl[(\bv\cdot\nabla)\bv-\Delta \bv+\nabla p\biggr]dt=\biggl[-\nabla\cdot(\nabla \bd \odot \nabla \bd)\biggr]dt+ Q(\bv)\,d\tilde{W} \text{ in } \mo\times (0,T]\label{F-LC-1}\\
		&	\nabla\cdot \bv=0 \text{ in }  \mo\times [0,T],\label{F-LC-2}\\
		& \int_\mo \bv(t,x) dx=0 \text{ for all  } t \in [0,T] \\
		&	\partial_t \bd+(\bv\cdot\nabla)\bd =\Delta \bd+|\nabla \bd|^2\bd+ (\bd \times \mathbf{h})\circ d\eta \text{ in } \mo\times (0,T]\label{F-LC-3}\\
		&	|\bd|^2=1 \text{ in } \mo\times [0,T].\label{F-LC-4}\\
		& \bv(0,\cdot)=\bv_0 \text{ and } \bd(0, \cdot)=\bd_0 \text{ in } \mo, \label{F-LC-5}
	\end{align}
\end{subequations}
where
we denote by $\mo$ the $d$-dimensional torus $[-\pi, \pi]^d$, $d=2,3$, the mapping
$\mathbf{h}:\mathbb{R}^d \to \mathbb{R}^3$ is a given function, $W$  a cylindrical Wiener process evolving on a separable Hilbert space $\rK_1$, $\eta$ is a one-dimensional standard Brownian motion, and $Q$ is a nonlinear map satisfying several conditions specified later on. 


Throughout this paper we assume that  $\bv, \bd, p$, as well as $\bh$ are $2\pi$-periodic in the following sense:
\begin{equation}\label{Eq:Periodic-BC}
	u(x+2\pi e_i)=u(x), \;\; u \in \{\bv,\bd,p, \bh  \},\;\; x\in \mathbb{R}^d,  i \in \{1, \ldots, d\},
\end{equation}
where $\{e_i, i=1, \ldots, d\}$  is the canonical basis of $\mathbb{R}^d$.
In what follows, when we refer to problem \eqref{F-LC}, we refer to the system of equations \eqref{F-LC} with the boundary condition given in \eqref{Eq:Periodic-BC}.

The system of SPDEs \eqref{F-LC} describes the dynamics of nematic liquid crystal with a stochastic perturbation.     Our investigation is motivated by the need for a mathematical analysis of the effect of the stochastic external perturbation on the dynamics of nematic liquid crystals. While the role of noise on the dynamics of $\bd$ has been the subject of numerous theoretical and experimental studies in physics, see, for instance, \cite{Horsthemke+Lefever-1984,San Miguel-1985,FS+MSanM}, in which  it is found that the time needed by the system to leave an unstable state diminishes in the presence of fluctuating magnetic fields, there are almost no rigorous mathematical results in this direction of research. The works
\cite{Horsthemke+Lefever-1984,San Miguel-1985,FS+MSanM} and the mathematical papers we cited earlier neglected either the effect of the velocity $\bv$ or the stochastic external perturbation, although, de Gennes and Prost \cite{Gennes}  noted that $\bv$ plays an essential role in the dynamics of $\bd$. It is this gap in knowledge that is the motivation for our mathematical study.
The current authors established in \cite{ZB+EH+PR-16} some existence, uniqueness and a maximum principle results for the stochastic version of a Ginzburg-Landau approximation of the system \eqref{F-LC}  without the sphere condition \eqref{F-LC-4}.

In this paper we study the local resolvability of  problem \eqref{F-LC}. Our result can be summarized as follows. Given a number $\alpha>\frac d2$ and a square integrable   $\sH^{\alpha}\times\bH^{\alpha+1}$-valued random variable $(\bv_0,\bd_0)$ we can find
a stopping time $\tilde{\tau}_\infty$ which can be approximated by an increasing sequence of stopping times $(\tau_m)_{m \in \mathbb{N}}$
and a unique local stochastic process $(\bv, \bd)=(\bv(t), \bd(t)), 0 \leq t< \tilde{\tau}_\infty$  satisfying the following conditions
\begin{enumerate}
	\item $\tilde{\tau}_\infty >0$ with positive probability,
	\item  $ (\bv(\cdot\wedge \tau_m),\bd(\cdot\wedge \tau_m))\in C([0,T];\sH^{\alpha} \times \bH^{\alpha+1})\cap \el^2(0,T; \bH^{\alpha+1} \times \bH^{\alpha+2}) $ for any  $m \in \mathbb{N}$, with probability $1$;
	\item  and for all $t\in (0,T]$ and $m \in \mathbb{N}$ we have $\mathbb{P}$-a.s.\ $\lvert \bd (t\wedge \tau_m, x)\rvert=1$  for all $x \in \MO$,
	\item the process $(\bv, \bd)=(\bv(t), \bd(t)), 0 \leq t< \tilde{\tau}_{\infty}$  is a unique  local solution to problem \eqref{F-LC}, see Definitions \ref{def-local solution} and \ref{def-maximal solution}.
\end{enumerate}
Moreover,  we  established  probabilistic lower  bounds on the lifespan $\tilde{\tau}_\infty$ of the local maximal solution.

These results extend to the stochastic case the local existence and uniqueness results for \eqref{erick-leslie} obtained for the deterministic model by Wang et al.\  in \cite{WZZ}. Our proof consists of two steps. In the first one, we apply earlier results obtained in \cite{ZB+EH+PR-16} to prove the existence and uniqueness of a maximal local solution satisfying the mild form of equations \eqref{F-LC-1}-\eqref{F-LC-3}. In the second one, we prove that when properly localised the local solution preserves the sphere condition \eqref{F-LC-4}.

The structure of the paper is as follows. In section \ref{sec-spaces} we present the main notation and standing assumptions we will be using in the whole paper. In section \ref{Sec-main-results}, we introduce the concept of a solution and state our main results. The proofs of the main theorems are given in section \ref{Proof-Local-Sol} and section \ref{Proof-of-Sphere}.

\section{Functional spaces and hypotheses}\label{sec-spaces}
We begin by introducing the necessary definitions of functional spaces frequently used in
this work.
\noindent  We denote by $\MO$ the $d$-dimensional torus  $d=2,3$. Functions defined on $\MO$ will be frequently identified with functions defined on the set $[-\pi, \pi]^d$ satisfying appropriate to their regularity periodic boundary conditions, for example, \eqref{Eq:Periodic-BC}.

Throughout this paper we denote by $L^p(\mo)$ and ${W}^{m,p}(\mo)$, $p\in [1,\infty]$, $m\in \mathbb{N}$, the Lebesgue and Sobolev spaces of real valued functions defined on $\MO$, see e.g.\  the monograph \cite{Temam_1983} by Temam (compare \cite{BBM_2014}). The corresponding spaces of  $\mathbb{R}^d$(or some cases $\mathbb{R}^3$)-valued functions, will be denoted by the black-board fonts, e.g. the space $L^p(\mo,\mathbb{R}^d)$  will be denoted by $\el^p(\mo)$.

For non-integer $r>0$ the Sobolev spaces ${H}^{r,p}(\mo)$ and $\mathbb{H}^{r,p}(\mo)$ are  defined by using the complex interpolation method. We will also use the notation $H^r(\mo):=W^{r,2}(\mo)$.
We simply skip the symbol  of the torus $\mo$, when there is no risk of ambiguity. For instance  we will write  $L^p$, resp.\ $\el^p$  or $\bw^{m,p}$ instead of $L^p(\mo)$, resp.\  $\el^p(\mo)$ or  $\bw^{m,p}(\mo)$.

Given two Banach spaces $K$ and $H$,  we denote by $\mathscr{L}(K,H)$ the space of bounded linear operators. For two Hilbert space $K$ and $H$ we denote by $\mathscr{L}_2(K,H)$ the Hilbert space
of all Hilbert-Schmidt operators from $K$ to $H$. For $K=H$ we just write $\mathscr{L}(K)$ instead of $\mathscr{L}(K,K)$.

Following \cite{Temam_1983} we also introduce the following spaces
\begin{align*}
	\el^2_{0}=& \biggl\{ \bu \in L^2(\mo,\mathbb{R}^d): \int_\MO \bu(x) dx=0    \biggr\}, \\
	\rH=&\biggl\{ \bu \in \el^2_{0}:  \divv \bu=0   \biggr\},\\
	\mathbb{H}^{r}_{\text{sol}}=&\rH \cap \mathbb{H}^{r}, \;\; r\in (0, \infty),\quad
	\rV=\mathbb{H}^{1}_{\text{sol}}.
\end{align*}
In the above formula, the divergence is understood in the weak sense.
Note that $\mathbb{H}^{0}_{\text{sol}}=\rH$.

In  \eqref{F-LC}, it is convenient to eliminate the pressure $p$ by applying the Helmholtz-Leray projector operator $\Pi:\el^2\to \rH$ which projects into divergence free vectors and annihilates gradients. One of the remarkable properties of $\Pi$ is that $\Pi\in \mathscr{L}(\mathbb{H}^{r},\mathbb{H}^{r}_\sol)$, $r>0$, see \cite{ZB+SC+MF}. We will frequently use this property without further notice.

Next, we define the Stokes operator, denoted by $\rA$,  which is an unbounded linear operator on $\mathrm{H}$, as follows.
\begin{equation}
	\label{def-A}
	\left\{
	\begin{array}{ll}
		D(\rA) &:= \mathbb{H}_{\sol}^2 \cr \rA u&:=- \Pi \Delta u, \, u \in
		D(\rA).
	\end{array}
	\right.
\end{equation}
It is well known that $\rA$ is a strictly positive self-adjoint operator in $\rH$ and that $D(\rA^{1/2})=\rV$. It is also true that $\rA$ is a strictly positive self-adjoint operator in every space $\mathbb{H}^{r}_{\text{sol}}$, $r>0$.

We will also need a version of the Laplace operator acting on $\mathbb{R}^3$-valued functions defined on $\mo$, i.e.
\begin{equation}
	\label{def-A_2}
	\left\{
	\begin{array}{ll}
		D(\mathrm{A_2}) &:= \mathbb{H}^2(\mo,\mathbb{R}^3), \cr \rA_2u&:=-  \Delta u, \, u \in
		D(\mathrm{A_2}).
	\end{array}
	\right.
\end{equation}
It is well known that $\rA_2$ is a non-negative self-adjoint operator in ${L}^2(\mo,\mathbb{R}^3)$. It is also true that $\rA_2$ is a non-negative self-adjoint operator in every space ${H}^r(\mo,\mathbb{R}^3)$, $r>0$.

It is well-known that 	$-\rA$ (resp. $-\rA_2$) is the infinitesimal  generator   analytic $C_0$-semigroup of contractions
on $\rH$, resp.\ $L^2(\mo,\mathbb{R}^3)$. These semigroups will be denoted by
$\{\mathbf{S}(t): t\geq 0\}$  and  $\{\mathbf{T}(t):t\geq 0\}$.
Moreover, for $s^\prime>s$ there exists a constant $M$ (depending on the difference $s^\prime-s$ and $p$)
such that we have (compare Lemma 1.2 in  the
Kato-Ponce's paper \cite{KP86})
\begin{equation}\label{semigruppo}
	\|\mathbf{S}(t)  \|_{\mathscr L(\mathbb{H}^{s}_{\mathrm{sol}};\mathbb{H}^{s^\prime}_{\mathrm{sol}})}\le M (1+t^{-(s^\prime-s)/2}), \;\;\; t>0,
\end{equation}
and
\begin{equation}\label{semigruppo-heat}
	\|\mathbf{T}(t)  \|_{\mathscr L(\mathbb{H}^{s};\mathbb{H}^{s^\prime})}\le M (1+t^{-(s^\prime-s)/2}), \;\;\; t>0.
\end{equation}
Let us note  the following  inequality involving fractional Sobolev norms.
\begin{align}
	\Vert fg\Vert_{\mathbb{H}^{s}}\le c_0(\Vert  f \Vert_{\el^ \infty} \Vert g\Vert_{\mathbb{H}^{s}} + \Vert  f \Vert_{\mathbb{H}^{s}}  \Vert g\Vert_{\el^ \infty}), \;\; f,g\in \el^ \infty \cap \mathbb{H}^{s}.
	\label{Comm-Est-2}
\end{align}

Now let $\alpha>d/2$. For any $\bu,\bv \in \sH^{\alpha}$ and $\bd, \bm \in \bH^{\alpha+1}$, where  now
$\bH^{\alpha+1}={H}^{\alpha+1}(\mo,\mathbb{R}^3)$,    we set
\begin{align}
	& B(\bu,\bv)=\Pi (\bu\cdot \nabla \bv),\\
	& M(\bd,\bm)=-\Pi(\nabla \cdot (\nabla \bd \odot \nabla \bm) ),\\
	& \tilde{B}(\bv,\bd)=\bv \cdot \nabla \bd.
\end{align}
Later on, we will state and prove few crucial properties of these nonlinear maps.
%

Let us fix $\bd\in \bH^{\alpha+1}$ and set
\begin{equation}\label{eqn-G} G(\bh)=\bh \times \bd, \;\; \;\bh\in \bH^{\alpha+1}.
\end{equation}
It is easy to see that $G\in \mathscr{L}(\bH^{\alpha+1})$.
Let us note that the  map $G^2 $, also an element of $\mathscr{L}(\bH^{\alpha+1})$, is of the following form
$$G^2(\bd)=G\circ G(\bd)=(\bd\times \bh)\times \bh. $$
Let $(\Omega,
\mathcal{F}, \mathbb{P})$ be a complete probability space equipped
with a filtration $\mathbb{F}=\{\mathcal{F}_t: t\geq 0\}$
satisfying the usual conditions.
Let $\tilde{W}=(\tilde{W}(t))_{t\ge 0}$ be a cylindrical Wiener process evolving on a separable Hilbert space $\rK_1$ such that it is formally written as a series
$$\tilde{ W}(t)= \sum_{k=1}^\infty \tilde{w}_k(t) \varphi_k, \; \; \forall t\ge 0,
$$
where $(\tilde{w}_k(t))_{k\in \mathbb{N}, t\ge0}$ is a family of i.i.d.  standard Brownian motions and $\{\varphi_k; k \in \mathbb{N}\}$ is an orthonormal basis of $\rK_1$.
The above series does not converge in $\rK_1$, but it does converges in a separable Hilbert space  $\tilde{\rK}_1$ such that the embedding $\rK_1 \subset \tilde{\rK}_1$ is Hilbert-Schmidt. It is well-known also that $\tilde{W}$ has a modification, still denoted by $\tilde{W}$, whose trajectories are  continuous  $\tilde{\rK}_1$-valued functions.  Let $\eta$ be a standard one dimensional Brownian motion and $\bh$ be a smooth vector fields.


We now introduce the assumptions on the coefficient $Q$ of the noise.
\begin{assum}\label{HYPO-ST}
	We fix  $\alpha>d/2$	and we assume that \[Q: \sH^{\alpha} \to \mathscr{L}_2(\rK_1,\sH^{\alpha})\] is a globally Lipschitz map. In particular,
	there exists $\ell_0\geq 0$
	such that
	\begin{equation*}
		\lve  Q(\bu)\rve^2_{\mathscr{L}_2(\rK_1,\sH^{\alpha})}\leq \ell_0 (1+\lve \bu \rve^2_{\sH^{\alpha}}),\;\; \mbox{ for any } \bu \in \sH^{\alpha}.
	\end{equation*}
\end{assum}
Hereafter \dela{for $\alpha>d/2$} we set
\begin{equation}\label{eqn-spaces}
	\begin{split}
		\mathbf{H}_\alpha=&\sH^{\alpha-1}\times \bH^{\alpha},\\
		\mathbf{V}_\alpha=&\sH^{\alpha}\times \bH^{\alpha+1},\\
		\mathbf{E}_\alpha=&\sH^{\alpha+1}\times \bH^{\alpha+2}.
	\end{split}
\end{equation}
The stochastic equations for nematic
liquid crystal \eqref{F-LC} can be rewritten as a stochastic evolution equation in the space $\mathbf{H}_\alpha$:
\begin{equation}\label{ABSTRACT-LC}
	d\y(t) +\mathbf{A}\y(t) dt+\mathbf{F}(\y(t))
	dt+\mathbf{L}(\y(t))dt=\mathbf{G}(\y(t)) d{W}(t),\quad t\ge 0,
\end{equation}
where, for  $\y=(\bv, \d)\in \mathbf{E}_\alpha$ and $k=(k_1,k_2)\in \rK:=\rK_1\times \err$,
we have
\begin{equation}\label{eqn-def-A-F}
	\A\y=\begin{pmatrix} \rA \bv  \\
		\rrA \bd
	\end{pmatrix},\;\;\quad
	\f(\y)=\begin{pmatrix} B(\bv,\bv)+M(\bd)\\
		\tilde{B}(\bv,\bd)+\lvert \nabla \bd \rvert^2 \bd
	\end{pmatrix},
\end{equation}
\begin{equation}\label{eqn-def-L-G}
	\mathbf{L}(\y)=\begin{pmatrix} 0\\ -\frac 1 2 G^2(\bd)
	\end{pmatrix}, \quad \quad\mathbf{G}(\y)k=\begin{pmatrix}Q(\bu)k_1\\
		G(\bd)k_2 \end{pmatrix}.
\end{equation}
The process $W$ is a cylindrical Wiener process on $\rK$ such that for any $t\ge0$
\begin{equation*}
	W(t)=\begin{pmatrix} \tilde{W}(t)\\
		\eta(t)
	\end{pmatrix},\quad t\ge 0
	.
\end{equation*}
\section{Existence and uniqueness of local maximal solution}\label{Sec-main-results}
We first recall several definitions and concepts  which are given in the  following notations/definitions and are borrowed
from \cite{Brz+Elw_2000} or \cite{Kunita-90}.
\begin{Def}(compare \cite[p.\ 45]{Kunita-90})
	For a probability  space  $(\Omega,\mathcal{F}, \mathbb{
		P})$      with  a given right-continuous filtration $
	\mathbb{F}=\big(\mathcal{F}_t\big)_{t\ge  0}$,   a  stopping time
	$\tau$ is called accessible iff there exists an  increasing
	sequence of stopping times  $\tau_n$ such that a.s. $\tau_n <
	\tau$ and $\lim_{n\to \infty} \tau_n =\tau$.
\end{Def}
\vspace*{4pt}\noindent\textbf{Notation.}  For $t\ge 0$  and a  stopping  time  $\tau$  we  set
\[ \Omega_t(\tau) :=
\{ \omega \in \Omega : t < \tau(\omega)\},
\]
\[
[0,\tau)\times \Omega :=
\{ (t,\omega) \in [0,\infty)\times \Omega: 0\le t < \tau(\omega)
\}.
\]

\begin{Def} If  $X$ is    a topological space, then
	an $X$-valued  process $\xi : [0,\tau) \times \Omega \to X$,      is
	{\sl admissible} iff
\begin{itemize}
%
		\item[(i) ] it is adapted, i.e.  $\xi|_{\Omega_t(\tau)}: \Omega_t(\tau) \to
		X$ is $\mathcal{F}_t$ measurable, for any $t\ge 0$; \item[(ii)]
		for almost all $\omega \in \Omega$, the function $[0,
		\tau(\omega))\ni t \mapsto \xi(t, \omega) \in X$ is continuous.
\end{itemize}
	We  will also use for an admissible process $\xi : [0,\tau) \times \Omega \to X$   the notation  $ \{\xi(t), t< \tau\}$ or $(\xi,\tau)$.
	
	
	A process $\xi : [0,\tau) \times \Omega \to X$ is  {\sl progressively
		measurable}  iff for any $t> 0$  the  map $$[0,t\wedge \tau)
	\times \Omega \ni  (s,\omega) \mapsto  \xi(s,\omega) \in  X$$ is
	$\mathcal{B}_{t\wedge \tau} \times \mathcal{F}_{t\wedge \tau}$ measurable.

	Two  processes $\xi_i: [0,\tau_i)   \times \Omega  \to X$,
	$i=1,2$ are called {\sl equivalent},   iff $\tau_1=\tau_2$  a.s. and  for any  $t>0$
	the  following holds
	\[
	\xi_1(\cdot,\omega)= \xi_2(\cdot,\omega) \mbox{ on } [0,t] \mbox{       	for a.a. $\omega \in \Omega_t(\tau_1)\cap \Omega_t(\tau_2)$.
	}
	\]
	We will use for two equivalent processes $\xi_1$ and $\xi_2$ the notation $(\xi_1,\tau_1) \sim
	(\xi_2,\tau_2)$)
\end{Def}

Note that if processes  $\xi_i  :  [0,\tau_i)  \times  \Omega \to
X$,  $i=1,2$ are admissible and for any $t>0$
$\xi_1(t)|_{\Omega_t(\tau_1)}= \xi_2(t)|_{\Omega_t(\tau_2)}$
a.s.\, then they are  equivalent.


We now define some concepts of solution to \eqref{ABSTRACT-LC}, see \cite[Def. 4.2]{Brz+Millet_2012} or \cite[Def.
2.1]{Mikulevicius}.

\begin{Def}\label{def-local solution} Let $\y_0$ be a $\mathbf{V}_\alpha$-valued  $\mathcal{F}_0$--measurable random variable  such that $\mathbb{E} \Vert \y_0\Vert^2_{\mathbf{V}_\alpha}<\infty$. A local mild
	solution to problem \eqref{ABSTRACT-LC} with initial condition $\y(0)=\y_0$ is a pair $(\y,\tau)$ such that
	\begin{enumerate}
		\item $\tau$ is an accessible $\mathbb{F}$-stopping time, \item
		$\y: [0,\tau)\times \Omega \to \mathbf{V}_\alpha$ is an admissible process, \item there
		exists an approximating sequence  $(\tau_m)_{m\in \mathbb{N}}$ of
		finite $\mathbb{F}$-stopping times  such that $\tau_m \toup \tau$
		a.s. and, for every $m\in \mathbb{N}$ and $t\ge 0$,
		we have
		\begin{eqnarray}
&&\hspace{-2truecm}\lefteqn{\mathbb{E}\Big(  \sup_{s\in [0,t\wedge \tau_m]} \Vert
				\y(s)\Vert^2_{\mathbf{V}_\alpha} +\int_0^{t\wedge \tau_m} \Vert \y(s)\Vert_{\mathbf{E}_\alpha}^2 \,
				ds\Big)<\infty,} \label{eq-locsol_01}
			\\
			\label{eq-locsol_01-b}\hspace{-1truecm} \y(t\wedge \tau_m)&=& \mathbb{S}(t\wedge
			\tau_m)\y_0-\int_0^{t\wedge \tau_m}
			\mathbb{S}(t\wedge\tau_m-s)[ \mathbf{F}(\y(s))+\mathbf{L}(\y(s)]\; ds\\
			\nonumber &+&\int_0^{t\wedge\tau_m}
			\mathbb{S}(t\wedge\tau_m-s)\mathbf{G}(\y(s))\,
			d{W}(s)  \mbox{ in }  \mathbf{H}_\alpha\; \mathbb{P}\text{-a.s.}
		\end{eqnarray}
		\item The stopped processes $\bd(\cdot \wedge \tau_m)$, $m\in \mathbb{N}$, satisfies: for all
		$t \in [0,T]$, $m \in \mathbb{N}$ , $\mathbb{P}$-a.s.\
		
		\begin{equation}\label{Sphere-contraint}
			\lvert \bd(t\wedge \tau_m,x,\omega)\rvert^2=1 ,
		\end{equation}
		for all $x\in \mo$.
		
	\end{enumerate}
\end{Def}
We also introduce the notion of maximal local solution and its lifespan.

\begin{Def}\label{def-maximal solution}
	Let us denote the family of all local  mild solution
	$(u,\tau)$ to  the problem \eqref{ABSTRACT-LC} by $\mathcal{ LS}$. For two
	elements $(u,\tau), (v,\sigma) \in \mathcal{ LS} $ we write
	$(u,\tau)\preceq (v,\sigma)$, iff $\tau \leq \sigma$ a.s.\ and
	$v_{\vert [0,\tau)\times \Omega} \sim u$. We write $(u,\tau)\prec
	(v,\sigma)$, iff $(u,\tau)\preceq (v,\sigma)$ and $\tau<\sigma$ with positive probability.
	If there exists a maximal element  $(u,\tau)$ in the set $(\mathcal{
		LS},\preceq)$, then it is 
	is called a maximal local  mild solution    to  the problem \eqref{ABSTRACT-LC}.
	If $(u, \tau)$ is a  maximal local mild solution to equation
	\eqref{ABSTRACT-LC}, then the stopping time $\tau$ is called its
	lifetime.
	
\end{Def}
\begin{Rem}
	\begin{mlist}
		\item  Note that if
		$(u,\tau)\preceq (v,\sigma)$ and $(v,\sigma)\preceq (u,\tau)$,
		then  $(u,\tau)\sim (v,\sigma)$.
		\item The pair $(\mathcal{ LS},\preceq)$ is a
		partially ordered set in which, according to the  Elworthy's Amalgamation
		Lemma, see \cite[Lemmata III 6A and
		6B]{Elw_1982},   every non-empty chain has a least upper bound.
		%
	\end{mlist}
	
\end{Rem}

Having defined our solution concept, we can now state and prove the existence of a maximal local solution for our model. We also give a lower estimate and a characterisation of the local solution's lifespan.
\begin{thm}\label{LC-Local-Sol}
    Let $d\in \{2,3\}$, $\alpha>d/2$, $\bh\in \bH^{\alpha+1}$. If Assumption \ref{HYPO-ST} is satisfied, then for all $\mathcal{F}_0$-measurable and square integrable $\sH^{\alpha}\times \bH^{\alpha+1}$-valued random variables  $\y_0=(\v_0,\d_0)$
    the problem \eqref{ABSTRACT-LC} for the stochastic liquid crystals has a unique local maximal strong solution $((\bv;\bd), \tilde{\tau}_\infty)$  satisfying the following properties:
	\begin{enumerate}
		\item given $R>0$ and  $\varepsilon >0$ there exists
		$\tau(\varepsilon,R)>0$,  such that for every
		$\mathcal{F}_0$-measurable $\sH^{\alpha}\times \bH^{\alpha+1}$-valued random variable $(\v_0,\d_0)$
		satisfying\break  $\mathbb{E}\Vert (\v_0,\d_0) \Vert^{2}_{\bH^{\alpha}\times \bH^{\alpha+1} } \leq R^{2}$, one has
		\[{\mathbb P}\big(\tilde{\tau}_\infty \geq \tau(\varepsilon,R)\big) \geq
		1-\varepsilon.\]
		\item \label{THM-ii} We also have
		\begin{align}
			\mathbb{P}\left(\{ \tilde{\tau}_\infty <\infty \}\cap \{ \lve  (\bv  , \bd) \rve_{C([0,T];\sH^{\alpha}\times \bH^{\alpha+1})  }<\infty  \} \right)=0,\label{MAX-P1}\\
			\limsup_{t\toup\tilde{\tau}_\infty } \lve \bv(t) \rve_{\sH^{\alpha}} + \lve \bd(t) \rve_{\bH^{\alpha+1}}=\infty \text{ a.s. on } \{\tilde{\tau}_\infty<\infty \}\label{MAX-P2}.
		\end{align}
	\end{enumerate}
	
\end{thm}
We will show in the next theorem that the local solution from Theorem \ref{LC-Local-Sol} satisfies \eqref{F-LC-4}.
\begin{thm}\label{Thm-Sphere-cond}
	Assume that all the assumption of Theorem \ref{LC-Local-Sol} are satisfied. Let $\y_0 =(\v_0,\bd_0)\in \mathbf{V}_\alpha$ such that
	$\lvert \bd_0(\omega,x)\rvert^2= 1$ for  all  $\omega\in \Omega$ and all $x \in  \MO$. Let $(\y;\tau)=((\v, \bd);\tau)$ be a local solution to \eqref{ABSTRACT-LC} and $(\tau_m)_{m \in \mathbb{N}}$ an increasing sequence of stopping times approximating $\tau$. Then, for all $m \in \mathbb{N}$ and $t\in (0,T]$,  $\mathbb{P}$-a.s.
	$\lvert \bd(t\wedge \tau_m,x,\omega)\rvert^2 = 1$ for all   $x\in  \MO$.
\end{thm}

\begin{Rem}\label{rem-future} We suspect that, if $d\in \{2,3\}$, $\alpha>d/2$, then under reasonable assumptions about the noise, there exists a local maximal solution for every initial data  $\y_0=(\v_0,\d_0) \in \sH^{\alpha-1}\times \bH^{\alpha}$. We also suspect that the existence of a local solution is mainly due to the mathematical analysis. We limited ourselves to the analysis of local solution as we were not able to derive proper estimates yielding global existence. We, however, have the conjecture that under smallness condition on the initial data one should be able to prove global existence of solution; this is the case for the deterministic model, see \cite{WZZ}.   These questions will be investigated in subsequent papers.
\end{Rem}

The proofs of these two theorems are given in sections \ref{Proof-Local-Sol} and \ref{Proof-of-Sphere}, respectively.

\section{Proof of Theorem \ref{LC-Local-Sol}}\label{Proof-Local-Sol}
In order to prove the results in Theorem \ref{LC-Local-Sol}   we will use the general results proved in \cite[Theorem 5.15 and 5.16]{ZB+EH+PR-16}. For this purpose we establish several crucial estimates for the nonlinear terms in \eqref{F-LC} in the following lemmata.
\begin{lem}\label{lem-Bilinears-Est}
	Assume that $\alpha > d/2$. Then, there exist  $\delta \in [0,1)$ and $C>0$ such that  for any $\bu\in \sH^{\alpha},\bv \in \sH^{\alpha+1}$ and $\bd, \bm \in \bH^{\alpha+1}$
	\begin{align}
		& \lVert B(\bu,\bv) \rVert_{\bH^{\alpha-1}}\le C\left(\lVert \bu \rVert_{\el^\infty} \lVert \bv \rVert_{\bH^{\alpha}}+ \lVert \bu \rVert_{\bH^{\alpha-1}}  \lVert \bv \rVert_{\bH^{\alpha+1}}^{\delta} \lVert \bv \rVert^{1-\delta}_{\bH^{\alpha}}\right),\label{1st-EST}\\
		&  \lVert\tilde{B}(\bv,\bd) \rVert_{\bH^{\alpha}} \le C \lVert \bv \rVert_{\bH^{\alpha}}\lVert \bd \rVert_{\bH^{\alpha+1}}, \label{2nd-EST}\\
		&   \lVert M(\bd, \bm)\rVert_{\bH^{\alpha-1}}          \le C \lVert  \bd \rVert_{\bH^{\alpha+1}} \lVert  \bm \rVert_{\bH^{\alpha+1}}.\label{3rd-EST}
	\end{align}
	
\end{lem}
\begin{proof}[Proof of Lemma \ref{lem-Bilinears-Est}]
	Let $\bu\in\sH^{\alpha} , \bv \in \sH^{\alpha+1}$ and $\bd, \bm \in \bH^{\alpha+1}$. In what follows we will denote by $C$  various generic constants not depending neither on  $\bu,\bv,\bd$ nor $\bm$.  By the inequality \eqref{Comm-Est-2}, we get
	\begin{equation*}
		\lVert \bu \cdot \nabla \bv \rVert_{\bH^{\alpha-1}} \le C \left( \lVert \bu \rVert_{\el^\infty} \lVert \nabla \bv \rVert_{\bH^{\alpha-1}}+ \lVert \bu \rVert_{\bH^{\alpha-1}} \lVert \nabla \bv \rVert_{\el^\infty}\right),
	\end{equation*}
	Since $\alpha >d/2$, one can find a positive constant $\delta\in (0,1)$ such that $\alpha-\delta > d/2$. Thus, by the Sobolev embedding $\bH^{\alpha-\delta}\subset \el^\infty $  and the Gagliardo-Nirenberg inequality we have
	\begin{equation}
		\lVert \nabla \mathbf{g} \rVert_{\el^\infty}\le  C \lVert \nabla \mathbf{g} \lVert_{\bH^{\alpha-\delta}} \le  C \lVert \nabla \mathbf{g} \rVert^{1-\delta}_{\bH^{\alpha-1}} \lVert \nabla \mathbf{g} \lVert^{\delta}_{\bH^{\alpha}},
	\end{equation}
	from which    we infer that
	\begin{equation*}
		\lVert \bu \cdot \nabla \bv \rVert_{\bH^{\alpha-1}} \le C\left(\lVert \bu \rVert_{\el^\infty} \lVert \bv \rVert_{\bH^{\alpha}}+ \lVert \bu \rVert_{\bH^{\alpha-1}}  \lVert \bv \rVert_{\bH^{\alpha+1}}^{\delta} \lVert \bv \rVert^{1-\delta}_{\sH^{\alpha}}\right) .
	\end{equation*}
	The first estimate in our lemma easily follows from this last line and the fact that (as we are on the torus) the Leray-Helmhotz projection operator $\Pi$ belongs to $\mathscr{L}(\bH^{\alpha-1}, \sH^{\alpha-1})$.
	
	We now prove the second estimate. As  $\alpha>d/2$, $\sH^{\alpha}$ is an algebra and  we can easily infer that
	\begin{align*}
		\lVert \bv \cdot \nabla \bd \rVert_{\sH^{\alpha}} \le  \lVert \bv \rVert_{\bH^{\alpha}} \lVert \bd \rVert_{\bH^{\alpha+1}},
	\end{align*}
	from which  the second estimate in our lemma easily follows.
	
	Now we deal with third estimate where the nonlinear map $M$ is involved.  Since $\Pi\in \mathscr{L}(\sH^{\alpha-1})$ we get
	\begin{align*}
		\lVert M(\bd, \bm)\rVert_{\bH^{\alpha-1}}\le & C \lVert \nabla \bd \odot \nabla \bm \rVert_{\sH^{\alpha}}.
	\end{align*}
	Using an argument similar to the proof of \eqref{2nd-EST} yields  \eqref{3rd-EST}.
\end{proof}

We will also need the following lemma.
\begin{lem}\label{lem-Nonlinear-Est}
	Let $\alpha > d/2$. Then, there exists a constant $C>0$ such that  for any $\bd, \bm \in \bH^{\alpha+1}$
		\begin{equation}\label{4th-EST}
		\begin{split}
			\lVert \lvert \nabla \bd \rvert^2 \bd - \lvert \nabla \bm \rvert^2 \bm \rVert_{\h^{\alpha}}\le  C \lVert \bd -\bm \rVert_{\bH^{\alpha+1}}[ \lVert \bd \rVert_{\bH^{\alpha+1}}  +\lVert \bm \rVert_{\bH^{\alpha+1}}] \lVert \bd \rVert_{\h^{\alpha}} \\
			+C \lVert \bm \rVert^2_{\bH^{\alpha+1}} \lVert \bd -\bm \rVert_{\sH^{\alpha}}
		\end{split}
	\end{equation}
	
\end{lem}
\begin{proof}[Proof of Lemma \ref{lem-Nonlinear-Est}]
	Let us fix  $\bd, \bm \in \bH^{\alpha+1}$. Again, since  $\bH^\alpha$ is an algebra, we easily deduce from  the inequality
	\begin{align*}
		&\lVert \lvert \nabla \bd \rvert^2 \bd - \lvert \nabla \bm \rvert^2 \bm \rVert_{\bH^{\alpha}}\\\le & \lVert [(\lvert \nabla \bd\lvert -\lvert \nabla \bm\lvert) (\lvert \nabla \bd\rvert + \lvert \nabla \bm\rvert)] \bd \rVert_{\sH^{\alpha}} + \lVert  \lvert \nabla \bm \rvert^2 (\bd-\bm)\rVert_{\sH^{\alpha}},
	\end{align*}
	the  inequality         \eqref{4th-EST}.
\end{proof}
We now are ready to embark on  the promised proof of Theorem \ref{LC-Local-Sol}.
\begin{proof}[Proof of Theorem \ref{LC-Local-Sol}]

	Since the maps $M$, $B$ and $\tilde{B}$ are bilinear, we infer from the Lemmata \ref{lem-Bilinears-Est} and \ref{lem-Nonlinear-Est} that the nonlinear term $\mathbf{F}$ defined in  \eqref{eqn-def-A-F} satisfies the following property: there exist two constants $\delta \in (0,1)$ and $C>0$ such that for any $\y_1,\y_2\in \mathbf{E}_\alpha$ we have
	\begin{equation}\label{nonlinear-F}
		\begin{split}
			\lVert \mathbf{F}(\y_1 ) -\mathbf{F}(\y_2)\rVert_{\mathbf{H}_\alpha}\le & C  \lVert \y_1 -\y_2\rVert_{\mathbf{V}_\alpha }\left(\lVert \y_1\rVert^{1-\delta}_{\mathbf{V}_\alpha} \lVert \y_2 \rVert^\delta_{\mathbf{E}_\alpha }+ \sum_{k=1}^2\Big[ \lVert \y_1 \rVert^{k}_{\mathbf{V}_\alpha} + \lVert \y_2 \rVert^{k}_{\mathbf{V}_\alpha}\Big] \right)\\
			& +
			C  \lVert \y_1 -\y_2\rVert^{1-\delta}_{\mathbf{V}_\alpha} \lVert \y_1 -\y_2\rVert^{\delta}_{\mathbf{E}_\alpha} \lVert \y_2\rVert_{\mathbf{V}_\alpha }.
		\end{split}
	\end{equation}
	From the definition \ref{eqn-G} of the map $G$ and the assumption on $\bh$ it follows  that $\mathbf{L}\in \mathscr{L}(\sH^{\alpha}\times \bH^{\alpha+1})$ from which,  along with \eqref{nonlinear-F}, we infer that $\mathbf{F}+\mathbf{L}$ satisfies Assumption \ref{assum-F} of Theorem \ref{thm_local} ( see also  \cite[Assumption 5.1]{ZB+EH+PR-16}).
	
	Because of Assumption \ref{HYPO-ST} and the fact that $G\in \mathscr{L}( \bH^{\alpha+1})$ it is clear that $\mathbf{G}$ satisfies Assumption \ref{assum-G} of Theorem \ref{thm_local}.
	
	Now, let $X_T$ be   the Banach space   \begin{equation}\label{eqn-X_T-0}
		X_T:= C([0,T];\mathbf{V}_\alpha) \cap L^2(0,T;\mathbf{E}_\alpha)
	\end{equation}
	with the norm defined by
	\begin{equation}\label{eqn-X_T-norm-0}
		\vert \bu\vert_{X_T}^2= \sup_{s \in [0,T]} \Vert \bu(s)\Vert^2_{\mathbf{V}_\alpha}+\int_0^T \vert \bu(s) \vert_{\mathbf{E}_\alpha}^2\, ds.
	\end{equation}
	It is know from \cite[Lemma 1.2]{Pardoux} or \cite[Lemma 1.5]{KP86} that the linear map  $\mathbb{S}\ast : \el^2(0,T;\h_{\alpha}) \to X_T$ defined by
	$$ (\mathbb{S}\ast f) (\,\cdot\, )= \int_0^{\,\cdot\,} \mathbb{S}(\,\cdot\,-s) f(s) ds, \quad f\in \el^2(0,T; \bH_{\alpha}), $$
	is bounded.
	
	It is also know, see \cite[Lemma 1.4]{Pardoux} or \cite[Chapitre2, Lemma 2.1]{Pardoux2}, that the linear map
	$\mathbb{S}\diamond : M^2(0,T;\mathscr{L}_2(\rK, \mathbf{V}_{\alpha}) ) \to  M^2(X_T)$  defined by
	$$ (\mathbb{S}\diamond g)(\,\cdot\,)=\int_0^{\,\cdot\,} \mathbb{S}(\,\cdot\,-s)g(s)dW(s), \;\; g\in M^2(0,T;\mathscr{L}_2(\rK, \mathbf{V}_{\alpha})),$$
	is also bounded.
	
	From the observations above, Assumption \ref{HYPO-ST} and the assumption on the initial data $\y_0$  we infer that the problem \eqref{ABSTRACT-LC} satisfies all the assumptions of Theorems \ref{thm_local} and \ref{thm_maximal-abstract} ( see also \cite[Theorem 5.15 and 5.16]{ZB+EH+PR-16}) from which we easily  complete the proof of the Theorem \ref{LC-Local-Sol}.
\end{proof}

\section{Proof of  Theorem  \ref{Thm-Sphere-cond} } \label{Proof-of-Sphere}
In this section we give the proof of the sphere constraint.

\begin{proof}[Proof of Theorem \ref{Thm-Sphere-cond}]
	The proof will be divided into two steps.
	
	Let
	$\varphi:\mathbb{R}\rightarrow [-1,0]$ be a $\mathcal{C}^\infty$ class  increasing function
	such that
	\begin{equation}
		\varphi(s) =\begin{cases}  -1 \;\text{ iff } s\in (-\infty, -2],\\
			0 \;\text{ iff } s\in [-1,+\infty).
		\end{cases}
	\end{equation}
	Let $\{\tilde{\varphi}_\ell: \ell\in \mathbb{N}\}$ and $\{\tilde{\phi}_\ell: \ell\in
	\mathbb{N}\}$ be two sequences of  function
	$\mathbb{R}$ defined by
	\begin{align}
		\tilde{\varphi}_\ell(a)=&\varphi(\ell a),\;\; a \in \mathbb{R}, \\
		\tilde{\phi}_\ell(a)=&a^2 {\varphi(\ell a)}, \;\; a \in \mathbb{R}.
	\end{align}
	We also  set	
	\begin{align}
		\varphi_\ell(\bd)=&\tilde{\varphi}_\ell (\lvert \bd \rvert^2-1) ,\; \bd \in \mathbb{R}^3,\\
		\phi_\ell(\bd)=&\tilde{\phi}_\ell(\lvert \bd \rvert^2-1), \;\bd\in
		\mathbb{R}^3.
	\end{align}
	Now, let $\alpha>\frac d2$ be a fixed number. For each $\ell \in \mathbb{N}$  we define  a function
	\begin{equation}
		\begin{split}
			\Psi_\ell&: \bH^{\alpha }\to \mathbb{R}
			\\
			\Psi_\ell(\bd)&=\lve \phi_\ell \circ \bd\rve_{\el^1} \label{regularize}
			=\int_\MO (\lvert \bd(x)\rvert^2-1)^2 [\varphi_\ell(\bd(x))] dx,\,\, \bd\in
			\bH^{\alpha}.
		\end{split}
	\end{equation}
	One can show that  since $\bH^{\alpha}\subset \el^\infty$ (as $\alpha>\frac d2$), the map   $\Psi_\ell$ is twice (Fr\'echet) differentiable\footnote{One might think that since $\Phi_\ell$ is well defined on the space $\bH^{1}$, it would also be twice differentiable 
		in $\bH^{1}$. However, for this to hold, we need to restrict it to the space $\bH^{\alpha}$ for $\alpha>\frac d2$ as in this case  $\bH^{\alpha}         \subset \el^\infty$. This is in fact related to the properties of Nemytski maps, see the papers by the first named authour \cite{Brz+Elw_2000} and \cite{Brz+Millet_2012}.}
	and its first and second derivatives satisfy,  for $\bd\in \bH^{\alpha} $  and $ \mathbf{k},\mathbf{f} \in \bH^{\alpha}$,
	\begin{equation}
		\begin{split}
			\Psi_\ell^\prime(\bd)(\mathbf{k} )= & 4\int_\MO \left(\lvert
			\bd (x)\rvert^2-1)\varphi_\ell(\bd(x)) [\bd(x)\cdot \mathbf{k}(x)]  \right)dx\\
			&\qquad 	+ 2\ell \int_\MO
			(\lvert \bd(x)\rvert^2-1)^2 \varphi^\prime (\ell (\lvert\bd(x)\rvert^2-1) ) (\bd(x)\cdot \mathbf{k}(x))dx ,
		\end{split}
		\label{Eq:1stDer}
	\end{equation}
	and
	\begin{equation}
		\label{Eq:2ndDer}
		\begin{split}
			\Psi^{\prime \prime}_\ell(\bd)(\mathbf{k},\mathbf{f})= 4\ell^2 \int_\MO \left[ (\lvert \bd(x)\rvert^2-1)^2 \varphi_\ell^{\prime\prime}(\ell(\lvert \bd(x)\rvert^2-1  )) (\bd(x)\cdot \mathbf{k}(x) ) (\bd(x) \cdot
			\mathbf{f}(x) )  \right]dx\\
			+16\ell \int_\MO \left[(\lvert \bd(x)\rvert^2-1)\varphi^\prime (\ell(\lvert \bd(x)\rvert^2 -1)) (\bd(x) \cdot \mathbf{k}(x) ) (\bd(x) \cdot
			\mathbf{f}(x) )\right]dx\\
			+ 8\int_\MO \biggl[\varphi_\ell(\bd(x))
			(\bd(x)\cdot \mathbf{k}(x) ) (\bd(x) \cdot \mathbf{f}(x)) \biggr]dx\\
			+ 2\ell \int_\MO \left[\lvert \bd(x)\rvert^2 -1)^2 \varphi^\prime (\ell (\lvert \bd(x) \rvert^2-1) )
			(\mathbf{k}(x)\cdot \mathbf{f}(x)  \right]  dx\\ +4 \int_\MO \left[\varphi_\ell(\bd(x))
			(\lvert \bd(x)\rvert^2-1) (\mathbf{k}(x)\cdot \mathbf{f}(x) )\right]dx.
		\end{split}
	\end{equation}
	In particular, if $\bd\in \bH^{\alpha}$ and $\mathbf{k},\mathbf{f} \in \bH^{\alpha}$ are such that
	\[\mathbf{k}(x)\perp \bd(x) \mbox{ and }\mathbf{f}(x)\perp \bd(x) \mbox{  for all  }x \in \MO,\]
	then 	
	\begin{equation}
		\Psi_\ell^\prime(\bd)(\mathbf{k})=0, \label{Eq:1stDerPerp}
	\end{equation}
	and
	\begin{equation}
		\begin{split}
			\Psi_\ell^{\prime\prime}(\bd)(\mathbf{k},\mathbf{f} )=& 4\int_\MO \left[(\lvert
			\bd(x) \rvert^2-1)\varphi_\ell(\bd(x)) (\mathbf{k}(x) \cdot \mathbf{f}(x) )\right] dx\\
			& \qquad +2\ell  \int_\MO \left[ (\lvert
			\bd(x) \rvert^2-1)^2 \varphi^\prime( \ell (\lvert\bd(x) \rvert^2  -1) ) ( \mathbf{k}(x) \cdot \mathbf{f}(x))  \right]dx.\label{Eq:2ndDerPerp}
		\end{split}
	\end{equation}
	Since the local solution $\bd$ given by Theorem \ref{LC-Local-Sol} satisfies the following integral equation in $\bH^{\alpha}$, for all  $t \in [0,T]$, all $m \in \mathbb{N}$, $\mathbb{P}$-a.s.
	\begin{equation*}
		\bd(t\wedge \tau_m)= \bd_0 + \int_0^{t\wedge \tau_m} (\Delta \bd(s)  + \lvert \nabla \bd(s)\rvert^2 \bd(s) -\bv(s) \cdot \nabla \bd(s)  ) ds  +\int_0^{t\wedge \tau_m}(\bd(s) \times \bh)\circ d\eta, 
	\end{equation*}
	it          	follows from  the It\^o formula,  see  \cite[Theorem I.3.3.2]{Pardoux} and \cite[Theorem 1]{GK_1982}, that  for any $m \in \mathbb{N}$,  \text{ for all } $t \in [0,T]$, $\mathbb{P}-a.s.$,
		\begin{equation*}
		\begin{split}
			\Psi_\ell(\bd)(t\wedge \tau_m)-\Psi_\ell(\bd)(0)\\=\int_0^{t\wedge \tau_m}\Psi^\prime_\ell (\bd)(s) \left(\Delta \bd  +\lvert \nabla \bd\rvert^2 \bd -\bv \cdot \nabla \bd  +\frac 12 G^2(\bd)\right)(s)ds
			\\ +\int_0^{t\wedge \tau_m}\frac 12 \Psi_\ell^{\prime\prime}(\bd)(s) (G(\bd)(s, G(\bd)(s) ds.
		\end{split}
	\end{equation*}
	Note that the  stochastic integral vanishes because $G(\bd(s,x))\perp\bd(s,x)$ for all $s \in [0,\tau)$ and $x \in \MO$.    	
	
	 Since $G^2(\bd)=(\bd \times \bh )\times \bh$ and $G(\bd)=\bd \times \bh $, we infer from \eqref{Eq:1stDer} and the identity
	$$-\lvert a \times b\rvert^2_{\mathbb{R}^3}=a\cdot \left((a\times
	b)\times b\right), a, b \in \mathbb{R}^3,$$  that
	\begin{equation*}
		\begin{split}
			\Psi^\prime(\bd)(G^2(\bd) )=-2\ell \int_\MO (\lvert\bd(x) \rvert^2-1 ) \varphi^\prime(\ell(\lvert \bd(x)\rvert^2-1 ) ) \lvert G(\bd(x))\rvert^2 dx\\
			-4 \int_\MO (\lvert \bd (x) \rvert^2-1)\varphi_\ell (\bd(x))\lvert G(\bd(x))\rvert^2 dx,
		\end{split}
	\end{equation*}
	which along with  the fact that $G(\bd(x)) \perp \bd(x)$ for any $x \in \MO$  and   \eqref{Eq:2ndDerPerp} we infer that
	\begin{equation*}
		\frac 1 2 \Psi_\ell^{\prime\prime}(G(\bd), G(\bd))+\frac 12
		\Psi_\ell^{\prime}(G^2(\bd))=0.
	\end{equation*}
	Hence, for every $m \in \mathbb{N}$, \text{ for all } $t \in [0,T]$, $\mathbb{P}$-a.s.,
	{\small\begin{equation}\label{approx}
		\Psi_\ell(\bd(t\wedge \tau_m))-\Psi_\ell(\bd(0))=\int_0^{t\wedge \tau_m}\Psi^\prime_\ell (\bd(s)) \left(\Delta \bd(s)  + \lvert \nabla \bd(s)\rvert^2 \bd(s) -\bv(s) \cdot \nabla \bd(s)  \right) ds.
	\end{equation}}Now, observe that from the assumptions on the function  $\varphi$ and the definition of the sequence $\tilde{\phi}_\ell, \ell \in \mathbb{N}$ we infer  that, with $a_{-}:=\max(-a,0)$, for any $a\in \mathbb{R}$,
	\begin{equation}\label{Eq:Limphil}
		\tilde{\phi}_\ell(a) \to  (a_-)^2  \text{ and }  \ell \varphi^\prime(\ell a) \to 0 \text{ as } \ell \to \infty.
	\end{equation}
	Observe also that there exists a constant $C>0$ such that for all $\ell \in \mathbb{N}$ and $a\in \mathbb{R}$
	\begin{equation}\label{Eq:Boundphil}
		\lvert \tilde{\phi}_\ell(a)\rvert \le C a^2 \text{ and } \lvert \ell \varphi^\prime(\ell a)\rvert\le C \vert a \rvert.
	\end{equation}
	Therefore we  infer from \eqref{Eq:Limphil}, \eqref{Eq:Boundphil} and the Lebesgue Dominated Convergence Theorem that for $\bd\in \bH^{\alpha}, \mathbf{k}\in \bH^{\alpha}$
	\begin{align*}
		\lim_{\ell \rightarrow \infty}\Psi_\ell(\bd)=&\lve \left(\lvert
		\bd\rvert^2-1\right)_{-}\rve^2,\\
		\lim_{\ell \rightarrow
			\infty}\Psi^\prime_\ell(\bd)(\mathbf{k})=&4 \int_\MO[ \left(\lvert
		\bd(x)\rvert^2-1\right)_{-}(\bd(x)\cdot \mathbf{k}(x))] \,dx.
	\end{align*}
	Hence, setting $y(t)=\lve \left(\lvert \bd(t)\rvert^2-1\right)_{-}\rve_{\el^2}^2$ and 
	letting $\ell \rightarrow \infty$ in
	\eqref{approx} we obtain that for every $m \in \mathbb{N}$, \text{ for all } $t \in [0,T]$, $\mathbb{P}$-a.s.,
	{\small\begin{equation*}
		\begin{split}
			y(t\wedge \tau_m)-y(0)+4\int_0^{t\wedge \tau_m} \biggl(\int_\MO \biggl[-\Delta \bd(s,x)  - \lvert \nabla \bd(s,x)\rvert^2 \bd(s,x) +\bv(s,x) \cdot \nabla \bd(s,x)  \biggr]\\
			\times \biggl[\bd(s,x)
			\left(\lvert \bd(s,x)\rvert^2-1\right)_{-}\biggr] dx \biggr)ds=0.
		\end{split}
	\end{equation*}}Using the identities
	\begin{align}
		\nabla\lvert \bd \rvert^2= 2 \nabla \bd \bd, \label{Indentity-nabla}\\
		\Delta \lvert \bd \rvert^2= 2\Delta \bd \cdot \bd + 2 \lvert \nabla \bd \rvert^2,\label{Identity-laplace}
	\end{align}
	we deduce that for every $m \in \mathbb{N}$, \text{ for all } $t \in [0,T]$, $\mathbb{P}$-a.s.,
	\begin{equation*}
		\begin{split}
			y(t\wedge \tau_m)-y(0)+2\int_0^{t\wedge \tau_m} \biggl(\int_\MO \biggl[-\Delta \lvert \bd(s,x)\rvert^2   - 2 \lvert \nabla \bd(s,x)\rvert^2 (\lvert \bd(s,x) \rvert^2-1) \\
			\qquad \qquad +\bv(s,x) \cdot \nabla \lvert \bd(s,x)\rvert^2  \biggr]\biggl[
			\left(\lvert \bd(s,x)\rvert^2-1\right)_{-}\biggr] dx \biggr)ds=0.
		\end{split}
	\end{equation*}
	Now observe that from the definition of $\zeta:= \left(\lvert \bd\rvert^2-1\right)_{-}$ we have, for $\bd \in \bH^{\alpha+2}$,
	\begin{equation*}
		\begin{split}
			&	\int_\MO \biggl[-\Delta \lvert \bd(x)\rvert^2   - 2 \lvert \nabla \bd(x)\rvert^2 (\lvert \bd(x) \rvert^2-1)\biggr]\zeta(x) dx\\
			&\qquad = -\int_\MO \left(\Delta\zeta(x) \cdot \zeta(x)- 2\mathds{1}_{\lvert \bd(x) \rvert^2 \le 1} \lvert \nabla \bd(x \rvert^2 \zeta^2(x)  \right)dx \\
			&\qquad \qquad	\ge \int_\MO \lvert \nabla \zeta(x) \rvert^2 dx.
		\end{split}
	\end{equation*}
	Observe also that since $\nabla \cdot \bv=0$ we have\footnote{Note that by \cite[Exercise 7.1.5, p 283]{Atkinson}, $\zeta \in H^1$ if
		$\bd\in \bH^1$.}
	\begin{equation*}
		\int_\MO \bv(x) \cdot \nabla\lvert \bd(x) \rvert^2 \zeta(x) dx=\int_\MO \bv(x) \cdot \nabla \zeta(x) \zeta(x) dx=0.
	\end{equation*}
	\noindent Bearing in mind the two remarks above, we infer that
	for every $m \in \mathbb{N}$,  $y(t\wedge \tau_m)$ satisfies \text{ for all } $t \in [0,T]$, $\mathbb{P}$-a.s.,
	\begin{equation*}
		y(t\wedge \tau_m)+2\int_0^{t\wedge \tau_m}  \lvert \nabla \zeta (s)\rvert_{{L^2}}^2 ds \le y(0).
	\end{equation*}
	Since the second term in the left hand side of the above
	inequality is positive and $y(0)=\lve (\lvert \bd_0\rvert^2-1)_{-}\rve^2$ and by
	assumption $\lvert \bd_0(x,\omega)\rvert^2=1$ for  all $x\in  \MO$ and $\omega\in \Omega$  we
	deduce  that, for every $m \in \mathbb{N}$,  \text{ for all } $t \in [0,T]$, $\mathbb{P}$-a.s.,
	\begin{equation*}
		y(t\wedge \tau_m)=0.
	\end{equation*}
	Since $\bH^{\alpha+1}\subset C^1(\MO)$ as $\alpha>\frac d2$,  we infer that
	for all $m \in \mathbb{N}$, $t\in  [0,T]$, $\mathbb{P}$-a.s.
	\begin{equation}\label{eqn-minus} (\lvert \bd(t \wedge \tau_m,x, \omega)\rvert^2 -1 )_{-}{=0} \mbox{ for all $x\in \MO$.}
	\end{equation}
	Thus, to complete the proof it is sufficient  to show that for all $m \in \mathbb{N}$, for all $t\in  [0,T]$ we have , $\mathbb{P}$-a.s.
	\begin{equation}\label{eqn-plus}
		(\lvert \bd(t\wedge\tau_m,x) \rvert^2-1)_{+}=0 \mbox{  for all $x\in \MO$.}
	\end{equation}
	For this purpose we set
	\begin{eqnarray*}\xi(t,x)&:=& \bigl( \lvert \bd (t,x) \rvert^2-1\bigr)_{+}, \;\;(t,x) \in [0,\tau)\times \MO,\\
		z(t)&=&\lVert \xi (t)\rVert^2_{\el^2}, \;\;\; t\in[0,T],
	\end{eqnarray*}
	and construct a sequence of functions $\Psi_\ell$ very similar to the one defined in
	\eqref{regularize}. First let us define an increasing function  $\varphi: \mathbb{R} \to [0,1]$ 
	belonging to     	$\mathcal{C}^\infty$ satisfying
	\begin{equation}
		\varphi(s) = \begin{cases}1 \text{ iff } s\in [2,\infty),\\
			0 \text{ iff } s\in (-\infty,1].
	\end{cases}  	\end{equation}
	Now, we replace in    the definition of $\Psi_\ell$ given by \eqref{regularize} the old function $\varphi_l$ by the function defined above.
	With this new definition we can show by arguing as before      that
	for $\bd\in \h^2$ and $ \bk \in \el^2$ we have
	\begin{align*}
		\lim_{m\rightarrow \infty}\Psi_\ell(\bd)=&\lve \left(\lvert
		\bd\rvert^2-1\right)_{+}\rve_{{L^2}}^2,\\
		\lim_{m\rightarrow
			\infty}\Psi^\prime_\ell(\bd)(\bk)=&4 \int_\MO[ \left(\lvert
		\bd(x)\rvert^2-1\right)_{+}\bd(x)\cdot \bk(x)] \,dx,
	\end{align*}
	and, for every $m \in \mathbb{N}$, \text{ for all } $t \in [0,T]$, $\mathbb{P}$-a.s.,
	{\small\begin{equation}\label{eqn-5.21}
		z(t\wedge \tau_m)-z(0)+2\int_0^{t\wedge \tau_m} \int_\MO \lvert \nabla \xi(s,x) \rvert^2 dx ds-4 \int_0^{t\wedge \tau_m } \int_\MO \lvert \nabla \bd(s,x) \rvert^2\lvert \xi(s,x)\rvert^2 dx ds =0.
	\end{equation}}Since the third term in the left hand side of the above identity is negative we cannot neglect this term. Before proceeding with the proof we observe that from the assumption $\alpha>d/2 $ and \eqref{eq-locsol_01} we infer that for any $\eps>0$ there exists a constant $N>0$ such that
	\begin{eqnarray*}
		\mathbb{P}(\Omega_{m,N})& \geq& 1-\eps, \;\;\; \mbox{ where}\\
		\Omega_{m,N} &=&\bigl\{\omega \in \Omega:\; \sup_{s\in [0,t\wedge \tau_m]}\lVert \nabla \bd(s)  \rVert_{\el^\infty}\le N \bigr\}.
	\end{eqnarray*}
	Let us observe that for all $m\in \mathbb{N}$, in view of \eqref{eqn-5.21},    we have  on $\Omega_{m,N}$
	\begin{equation}
		\begin{split}
			z(t\wedge \tau_m)-z(0)+2\int_0^{t\wedge \tau_m} \int_\MO \lvert \nabla \xi(s,x) \rvert^2 dx ds&\leq 4N^2 \int_0^{t\wedge \tau_m }\int_\MO \vert \xi(s,x)\vert^2 dx ds\\
			& \leq 4N^2 \int_0^{t\wedge \tau_m } z(s) ds.
		\end{split}
	\end{equation}
	Taking the expectation (over the set $\Omega_{m,N}$), because for a nonnegative function $z$,\break $\int_0^{t \wedge \tau} z(s)\,ds \leq \int_0^{t } z(s\wedge \tau)\,ds$,  from the above inequality we get
	\begin{equation}
		\mathbb{E}\bigl[ z(t\wedge \tau_m)1_{\Omega_{m,N}} \bigr]
		\leq \mathbb{E}\bigl[ z(0)  1_{\Omega_{m,N}} \bigr] + 4N^2 \int_0^{t } \mathbb{E}\bigl[ z(s \wedge \tau_m) 1_{\Omega_{m,N}}\bigr] \, ds.
	\end{equation}
	Applying  the Gronwall Lemma we infer that
	\begin{equation*}
		\mathbb{E}\bigl[ z(t\wedge \tau_m)1_{\Omega_{m,N}} \bigr]   \le    \mathbb{E}\bigl[ z(t\wedge \tau_m)1_{\Omega_{m,N}} \bigr]  e^{4N^2 T}=0, \;\; t \in [0,T].
	\end{equation*}
	Hence we infer that  $1_{\Omega_{m,N}}(\lvert \bd (t\wedge \tau_m)\rvert^2-1)_{+}=0$ for every $t \in [0,T]$, $\mathbb{P}$-a.s. and therefore
	we deduce that  for every $t \in [0,T]$ and
	for every $\eps>0$
	\begin{equation*}
		\mathbb{P}\left((\lvert\bd t\wedge \tau_m \rvert^2-1)_{+}=0\right)\ge 1-\eps.
	\end{equation*}
	From this last estimate and the first part of the proof infer that for all $t\in  [0,T]$, $m \in \mathbb{N}$,  $\mathbb{P}$-a.s.
	\begin{equation*}
		(\lvert \bd(t\wedge \tau_m,x,\omega) \rvert-1)_{+}=(\lvert \bd(t\wedge\tau_m,x,\omega) \rvert-1)_{-}=0 \text{ for all $x\in \MO$,}
	\end{equation*}
	which implies \eqref{Sphere-contraint}.
\end{proof}

\appendix

\section{Local strong solution for an abstract stochastic evolution equation}\label{ABST-STRONG}
The goal of this section is to recall general results about the existence of a local and maximal solution to an abstract stochastic partial differential equation with locally
Lipschitz continuous coefficients. These results were proved in  \cite{ZB+EH+PR-16} utilising some
truncation and fixed point methods. The proofs are highly technical, and hence we refer the reader to \cite{ZB+EH+PR-16} for the details.

To start with let us fix some notations and assumptions.
Let $V$,  $E$ and $H$ be separable Banach spaces such that
$E\subset V$ continuously. We denote the norm in $V$ by $\Vert
\cdot \Vert$ and we put
\begin{equation}\label{eqn-X_T}
	X_T:= C([0,T];V) \cap L^2(0,T;E)
\end{equation}
with the norm $\vert \cdot \vert_{X_T}$ satisfying
\begin{equation}\label{eqn-X_T-norm}
	\vert u\vert_{X_T}^2= \sup_{s \in [0,T]} \Vert u(s)\Vert^2+\int_0^T \vert u(s) \vert_E^2\, ds.
\end{equation}
Let $F$ and $G$ be two nonlinear mappings satisfying the following
sets of conditions.
\begin{assum}\label{assum-F}
	Suppose that $F: E \to H$ is such that
	$F(0)=0$ and there exist $p ,q \geq 1$, $\alpha, \gamma \in [0,1)$ and $C>0$
	such that
	\begin{equation}\label{eqn-local Lipschitz-F}
		\begin{split}
			\vert F(y)-F(x) \vert_H \leq C \Big[ \Vert y-x\Vert \Vert
			y\Vert^{p-\alpha} \vert y\vert_E^\alpha + \vert y-x\vert_E^\alpha
			\Vert y-x\Vert^{1-\alpha} \Vert x\Vert^p\Big]\\
			+ C \Big[ \Vert y-x\Vert \Vert
			y\Vert^{q-\gamma} \vert y\vert_E^\gamma + \vert y-x\vert_E^\gamma
			\Vert y-x\Vert^{1-\gamma} \Vert x\Vert^q\Big],
		\end{split}
	\end{equation}
	for any $x, y\in E$.
\end{assum}
Let $\rK$ be a separable Hilbert space and $\mathscr{L}_2(\rK, V)$ the space of Hilbert-Schmidt operators from $\rK$ onto $V$. For the sake of simplicity we  denote by $\lVert \cdot \rVert_{\mathscr{L}_2}$ the norm in $\mathscr{L}_2(\rK, V)$.
\begin{assum}\label{assum-G}
	Assume that $G: E \to \mathscr{L}_2(\rK,V)$  such that $G(0)=0$ and there exists $k
	\geq 1$, $\beta \in [0,1)$ and $C_G>0$ such that
	\begin{equation}\label{eqn-local Lipschitz-G}
		\Vert G(y)-G(x) \Vert_{\mathscr{L}_2} \leq C_G \Big[ \Vert y-x\Vert \Vert
		y\Vert^{k-\beta} \vert y\vert_E^\beta + \vert y-x\vert_E^\beta
		\Vert y-x\Vert^{1-\beta} \Vert x\Vert^k\Big],
	\end{equation}
	for any $x, y\in E$.
\end{assum}
Let $(\Omega,
\mathcal{F}, \mathbb{P})$ be a complete probability space equipped
with a filtration $\mathbb{F}=\{\mathcal{F}_t: t\geq 0\}$
satisfying the usual condition.   By $M^2(X_T)$ we denote
the space of all progressively measurable $E$-valued processes
whose trajectories belong to $X_T$ almost surely endowed with a
norm $\vert \cdot \vert_{M^2(X_T)}$ satisfying
\begin{equation}\label{eqn-M^2X_T}
	\vert u \vert_{M^2(X_T)}^2 = \mathbb{E}\Big[ \sup_{s \in [0,T]}
	\Vert u(s)\Vert^2+\int_0^T \vert u(s) \vert_E^2\, ds\Big].
\end{equation}
Let us also formulate the following assumptions.
\begin{assum}\label{assum-01}
	Suppose that the embeddings $E\subset V \subset H$ are continuous. Consider (for
	simplicity) a one-dimensional Wiener process $W=\{W(t):t\ge 0\}$, \textit{i.e.}, $\rK=\mathbb{R}$.

	Assume that $\{S(t):t\in [0,\infty)\}$, is a family of bounded linear operators on the space $H$ such that there exists two positive constants $C_1$ and $C_2$ with the following properties:
	\\
	(i) For every $T>0$ and every  $f\in L^2(0,T;H)$ a function
	$u=S\ast f$ defined by
	\[ u(t)=\int_0^T S(t-r) f(r)\, dt, \;\; t \in [0,T],\]
	belongs to $X_T$ and
	\begin{equation}\label{ineq-dc}
		\vert u \vert _{X_T}\leq C_1 \vert f \vert _{L^2(0,T;H)}.
	\end{equation}
	(ii) For every $T>0$ and every process  $\xi\in M^2(0,T;\mathscr{L}_2(\rK, V))$ a
	process $u=S \diamond \xi$ defined by
	\[ u(t)=\int_0^T S(t-r) \xi(r)\, dW(r), \;\; t \in [0,T]\]
	belongs to $M^2(X_T)$ and
	\begin{equation*}
		\vert u \vert _{M^2(X_T)}\leq C_2 \vert \xi \vert _{M^2(0,T;\mathscr{L}_2(\rK, V))}.
	\end{equation*}
	(iii) For every $T>0$ and every $u_0\in V$, a function $u=Su_0$
	defined by
	\[ u(t)= S(t)u_0,  \;\; t \in [0,T]\]
	belongs to $X_T$. Moreover, for every $T_0>0$ there exist $C_0>0$
	such that for all $T\in (0,T_0]$,
	\begin{equation}\label{ineq-dc-2}
		\vert u \vert _{X_T}\leq C_0 \Vert u_0 \Vert.
	\end{equation}
\end{assum}
Now let us consider a semigroup $\{S(t):t\in [0,\infty)\}$ as above
and the abstract SPDEs
\begin{equation}\label{ABS-SPDE-1}
	u(t)=S(t)u_0+\int_0^t S(t-s) F(u(s)) ds+\int_0^t S(t-s) G(u(s))
	dW(s),\;\; \mbox{ for any }t>0,
\end{equation}
which is a mild version of the problem
\begin{equation}\label{ABS-SPDE-strong}
	\left\{\begin{array}{rl} du(t)&= Au(t)\,dt+  F\big(u(t)\big)\, dt+
		G\big(u(t)\big)
		dW(t),\;\;t>0,\\
		u(0)&=u_0.
	\end{array}
	\right.
\end{equation}
Here $A$ is the infinitesimal generator of the semigroup $\{S(t):t\ge 0\}$.

We will not recall the definitions of local and maximal solutions since they are  the same as the ones introduced definition \ref{def-local solution} and definition \ref{def-maximal solution}.  We directly give the main theorems that are of interest to us. The first one is about the existence and uniqueness of a local solution and a probabilistic lower bound of the solution's lifespan.
\begin{thm}\label{thm_local} Suppose that Assumption \ref{assum-F}, Assumption \ref{assum-G}, and Assumption \ref{assum-01} hold. Then for every
	$\mathcal{F}_{0}$-measurable $V$-valued square integrable random
	variable $u_0$    there exists  a  local process $u=\big(u(t),
	t\in[0,T_1) \big) $  which is the
	unique local mild solution to our  problem.
	Moreover,  given $R>0$ and  $\varepsilon >0$ there exists
	a stopping time $\tau(\varepsilon,R)>0$,  such that for every
	$\mathcal{F}_0$-measurable $V$-valued random variable $u_0$
	satisfying  $\mathbb{E}\Vert u_0 \Vert^{2} \leq R^{2}$, one has
	\[{\mathbb P}\big(T_1\geq \tau(\varepsilon,R)\big) \geq
	1-\varepsilon.\]
\end{thm}
The next result is about the existence and uniqueness of a maximal solution and the characterization of its lifespan.
\begin{thm}\label{thm_maximal-abstract}
	For every      $u_0\in L^2(\Omega,\mathcal{F}_0,V)$,
	the process $u=(u(t)\, ,\, t<  \tau_\infty) $ defined above
	is the unique local maximal solution to our equation.
	Moreover,
	$ {\mathbb P }\big(\{\tau_\infty <\infty\} \cap \{\sup_{
		t<\tau_\infty} |u(t)|_{V}<\infty \}\big)=0$ and on $\{
	\tau_\infty<\infty \}$, $\limsup_{t\to \tau_\infty} |u(t)|_{V} =
	+\infty$ a.s.
\end{thm}
The proofs of both theorems are highly nontrivial and technical, we refer to \cite[Section 5]{ZB+EH+PR-16} for the details.

\medskip
Received for publication July  2018.
\medskip


\begin{thebibliography}{99}

\bibitem{Atkinson} (MR2511061) [10.1007/978-1-4419-0458-4]
\newblock K.~Atkinson and W.~Han,
\newblock \emph{Theoretical Numerical Analysis. A Functional Analysis Framework,}
\newblock Third edition. Volume \textbf{39} of {Texts in Applied Mathematics}, Springer, Dordrecht, 2009.

\bibitem{Prohl} (MR2399392) [10.1137/07068254X]
\newblock R.~Becker, X. Feng and A.~Prohl,
\newblock \doititle{Finite element approximations of the Ericksen-Leslie model for nematic liquid crystal flow,}
\newblock \emph{SIAM J. Numer. Anal.}, \textbf{46} (2008), 1704--1731.

\bibitem{BBM_2014} (MR3274888) [10.1007/s40072-014-0041-7]
\newblock H. Bessaih, Z. Brze{\'z}niak and A. Millet,
\newblock \doititle{Splitting up method for the 2D stochastic Navier-Stokes equations,}
\newblock \emph{Stoch. Partial Differ. Equ. Anal. Comput.}, \textbf{2} (2014), 433--470.

\bibitem{ZB+SC+MF} (MR3383342) [10.1007/s00440-014-0584-6]
\newblock Z.~Brze\'zniak, S.~Cerrai and M.~Freidlin,
\newblock \doititle{Quasipotential and exit time for 2D Stochastic Navier-Stokes equations driven by space time white noise,}
\newblock \emph{Probab. Theory Related Fields}, \textbf{162} (2015), 739--793.

\bibitem{Brz+Elw_2000} (MR1784435)
\newblock Z. Brze\'zniak and K. D. Elworthy,
\newblock {Stochastic differential equations on Banach manifolds},
\newblock \emph{Methods Funct. Anal. Topology}, \textbf{6} (2000), 43--84.

\bibitem{ZB+BF-16} (MR3607727) [10.1007/s40072-016-0081-2]
\newblock Z.~Brze\'zniak and B.~Ferrario,
\newblock \doititle{A note on stochastic Navier--Stokes equations with not regular multiplicative noise},
\newblock \emph{Stoch. Partial Differ. Equ. Anal. Comput.}, \textbf{5} (2017), 53--80.

\bibitem{ZB+EH+PR-16}
\newblock Z. Brze\'zniak, E. Hausenblas and P. Razafimandimby,
\newblock Some results on the penalised nematic liquid crystals driven by multiplicative noise,
\newblock ArXiv preprint, \arXiv{1310.8641}, (2016), 65 pages.

\bibitem{Brz+Millet_2012} (MR3232027) [10.1007/s11118-013-9369-2]
\newblock Z.~Brze{\'z}niak and A. Millet,
\newblock \doititle{On the stochastic Strichartz estimates and the stochastic nonlinear Schr\"odinger equation on a compact Riemannian manifold,}
\newblock \emph{Potential Anal.}, \textbf{41} (2014), 269--315.

\bibitem{Cavaterra} (MR3045633) [10.1016/j.jde.2013.03.009]
\newblock C.~Cavaterra, R.~Rocca and H.~Wu,
\newblock \doititle{Global weak solution and blow-up criterion of the general Ericksen-Leslie system for nematic liquid crystal flows},
\newblock \emph{J. Differential Equations}, \textbf{255} (2013), 24--57.

\bibitem{Chandrasekhar}
\newblock S.~Chandrasekhar,
\newblock \emph{Liquid Crystals},
\newblock Cambridge University Press, 1992.

\bibitem{Rojas-Medar2} (MR2279252) [10.1007/s00033-005-0038-1]
\newblock B. Climent-Ezquerra, F. Guill\'en-Gonz\'alez and M. A. Rojas-Medar,
\newblock \doititle{Reproductivity for a nematic liquid crystal model},
\newblock \emph{Z. Angew. Math. Phys.}, \textbf{57} (2006), 984--998.

\bibitem{Climent} (MR3148841) [10.1017/S0956792513000338]
\newblock B. Climent-Ezquerra and F. Guill\'en-Gonz\'alez,
\newblock \doititle{A review of mathematical analysis of nematic and smectic-A liquid crystal models},
\newblock \emph{European J. Appl. Math.}, \textbf{25} (2014), 133--153.

\bibitem{Coutand-Shkoller} (MR1873808) [10.1016/S0764-4442(01)02161-9]
\newblock D. Coutand and S. Shkoller,
\newblock \doititle{Well-posdness of the full Ericksen-Leslie Model of nematic liquid crystals,}
\newblock \emph{C.R. Acad. Sci. Paris. S\'erie I}, \textbf{333} (2001), 919--924.

\bibitem{Dai+Schonbeck_2014} (MR3257634) [10.1137/120895342]
\newblock {M. Dai and M. Schonbek,}
\newblock \doititle{Asymptotic behavior of solutions to the liquid crystal system in $H^m(R^3)$},
\newblock \emph{SIAM J. Math. Anal.}, \textbf{46} (2014), 3131--3150.

\bibitem{Gennes}
\newblock P.~G.~ de Gennes and J.~Prost,
\newblock \emph{The Physics of Liquid Crystals},
\newblock Clarendon Press, Oxford, 1993.

\bibitem{Elw_1982} (MR675100)
\newblock K. D. Elworthy,
\newblock \emph{Stochastic Differential Equations on Manifolds},
\newblock London Math. Soc. LNS v \textbf{70}, Cambridge University Press, 1982.

\bibitem{Ericksen} (MR0158610) [10.1122/1.548883]
\newblock J.~L.~Ericksen,
\newblock \doititle{Conservation laws for liquid crystals},
\newblock \emph{Trans. Soc. Rheology}, \textbf{5} (1961), 23--34.

\bibitem{Gal16} (MR3275218) [10.1007/s00332-014-9211-z]
\newblock C. G. Gal and T. T. Medjo,
\newblock \doititle{On a regularized family of models for homogeneous incompressible two-phase flows},
\newblock \emph{J. Nonlinear Sci.}, \textbf{24} (2014), 1033--1103.

\bibitem{Grasselli} (MR3048212) [10.1137/120866476]
\newblock M.~ Grasselli and H.~Wu,
\newblock \doititle{Long-time behavior for a hydrodynamic model on nematic liquid crystal flows with asymptotic stabilizing boundary condition and external force},
\newblock \emph{SIAM J. Math. Anal.}, \textbf{45} (2013), 965--1002.

\bibitem{GK_1982} (MR665398) [10.1080/17442508208833202]
\newblock I. Gy\"ongy and N. V. Krylov,
\newblock \doititle{On stochastics equations with respect to semimartingales. II. It\^o formula in Banach spaces,}
\newblock \emph{Stochastics}, \textbf{6} (1981/82), 153--173.

\bibitem{Haroske} (MR2375667)
\newblock D.~D.~Haroske and H.~Triebel,
\newblock \emph{Distributions, Sobolev Spaces, Elliptic Equations,}
\newblock EMS Textbooks in Mathematics. European Mathematical Society, Z\"urich, 2008.

\bibitem{Hieber-16} (MR3465380) [10.1016/j.anihpc.2014.11.001]
\newblock M.~Hieber, M.~Nesensohn, J.~Pr\"uss and K.~Schade,
\newblock \doititle{Dynamics of nematic liquid crystal flows: The quasilinear approach},
\newblock \emph{Ann. Inst. H. Poincar\'e Anal. Non Lin\'eaire}, \textbf{33} (2016), 397--408.

\bibitem{Hieber-17} (MR3713532) [10.1007/s00208-016-1453-7]
\newblock M. Hieber and J. Pr\"uss,
\newblock \doititle{Dynamics of the Ericksen-Leslie equations with general Leslie stress I: the incompressible isotropic case},
\newblock \emph{Math. Ann.}, \textbf{369} (2017), 977--996.

\bibitem{Hong} (MR2745194) [10.1007/s00526-010-0331-5]
\newblock M.-C. Hong,
\newblock \doititle{Global existence of solutions of the simplified Ericksen-Leslie system in dimension two,}
\newblock \emph{Calculus of Variations,} \textbf{40} (2011), 15--36.

\bibitem{Hong12} (MR2964608) [10.1016/j.aim.2012.06.009]
\newblock M.-C. Hong and Z.~Xin,
\newblock \doititle{Global existence of solutions of the liquid crystal flow for the Oseen-Frank model in $\err^2$},
\newblock \emph{Adv. Math.}, \textbf{231} (2012), 1364--1400.

\bibitem{Hong14} (MR3208809) [10.1080/03605302.2013.871026]
\newblock M.-C. Hong, J. Li and Z. Xin,
\newblock \doititle{Blow-up criteria of strong solutions to the Ericksen-Leslie system in $\err^3$},
\newblock \emph{Comm. Partial Differential Equations}, \textbf{39} (2014), 1284--1328.

\bibitem{Horsthemke+Lefever-1984} (MR724433)
\newblock W.~Horsthemke and R.~Lefever,
\newblock \emph{Noise-induced Transitions. Theory and Applications in Physics, Chemistry, and Biology,}
\newblock Springer Series in Synergetics, \textbf{15}. Springer-Verlag, Berlin, 1984.

\bibitem{Huang14} (MR3238531) [10.1007/s00220-014-2079-9]
\newblock J. Huang, F. Lin, Fanghua and C. Wang,
\newblock \doititle{Regularity and existence of global solutions to the Ericksen-Leslie system in $\err^2$},
\newblock \emph{Comm. Math. Phys.}, \textbf{331} (2014), 805--850.

\bibitem{Huang16} (MR3509000) [10.1007/s00205-016-0983-1]
\newblock T.~Huang, F. Lin, C. Liu and C. Wang,
\newblock \doititle{Finite time singularity of the nematic liquid crystal flow in dimension three},
\newblock \emph{Arch. Ration. Mech. Anal.}, \textbf{221} (2016), 1223--1254.

\bibitem{KP88} (MR951744) [10.1002/cpa.3160410704]
\newblock T.~Kato and G.~Ponce,
\newblock \doititle{Commutator estimates and the Euler and Navier-Stokes equations},
\newblock \emph{Comm. Pure Appl. Math.,} \textbf{41} (1988), 891--907.

\bibitem{KP86} (MR864654) [10.4171/RMI/26]
\newblock T.~Kato and G.~Ponce,
\newblock \doititle{Well posedness of the Euler and Navier--Stokes equations in the Lebesgues spaces $L^p_s(\err^2)$,}
\newblock \emph{Rev. Mat. Iberoam.}, \textbf{2} (1986), 73--88.

\bibitem{Kunita-90} (MR1070361)
\newblock H. Kunita,
\newblock \emph{Stochastic Flows and Stochastic Differential Equations},
\newblock Cambridge University Press, 1990.

\bibitem{Leslie} (MR1553506) [10.1007/BF00251810]
\newblock F.~M.~Leslie,
\newblock \doititle{Some constitutive equations for liquid crystals},
\newblock \emph{Arch. Rational Mech. Anal.}, \textbf{28} (1968), 265--283.

\bibitem{Lin-Liu} (MR1329830) [10.1002/cpa.3160480503]
\newblock F.-H. Lin and C. Liu,
\newblock \doititle{Nonparabolic dissipative systems modeling the flow of liquid crystals,}
\newblock \emph{Communications on Pure and Applied Mathematics}, \textbf{48} (1995), 501--537.

\bibitem{Lin-Liu2} (MR1784963) [10.1007/s002050000102]
\newblock F.-H. Lin and C. Liu,
\newblock \doititle{Existence of solutions for the Ericksen-Leslie system,}
\newblock \emph{Arch. Rational Mech. Anal.,} \textbf{154} (2000), 135--156.

\bibitem{Lin-Wang14} (MR3273501) [10.1098/rsta.2013.0361]
\newblock F. Lin and C. Wang,
\newblock \doititle{Recent developments of analysis for hydrodynamic flow of nematic liquid crystals},
\newblock \emph{Philos. Trans. R. Soc. Lond. Ser. A Math. Phys. Eng. Sci.}, \textbf{372} (2014), 20130361, 18 pp.

\bibitem{Lin-Wang} (MR2745211) [10.1007/s11401-010-0612-5]
\newblock F. Lin and C. Wang,
\newblock \doititle{On the uniqueness of heat flow of harmonic maps and hydrodynamic flow of nematic liquid crystals,}
\newblock \emph{Chinese Annals of Mathematics, Series B.}, \textbf{31} (2010), 921--938.

\bibitem{Lin+Wang16} (MR3518239) [10.1002/cpa.21583]
\newblock F. Lin and C. Wang,
\newblock \doititle{Global existence of weak solutions of the nematic liquid crystal flow in dimension three},
\newblock \emph{Comm. Pure Appl. Math.}, \textbf{69} (2016), 1532--1571.

\bibitem{Lin-Lin-Wang} (MR2646822) [10.1007/s00205-009-0278-x]
\newblock F. Lin, J. Lin and C. Wang,
\newblock \doititle{Liquid crystals in two dimensions,}
\newblock \emph{Arch. Rational Mech. Anal.}, \textbf{197} (2010), 297--336.

\bibitem{Liu-Walkington1} (MR1740379) [10.1137/S0036142997327282]
\newblock C.~Liu and N. J.~Walkington,
\newblock \doititle{Approximation of liquid crystal flows,}
\newblock \emph{SIAM J. Numer. Anal.}, \textbf{37} (2000), 725--741.

\bibitem{Mikulevicius} (MR2529962) [10.1137/0807433747]
\newblock R.~Mikulevicius,
\newblock \doititle{On strong $\h^{1}_2$-solutions of stochastic Navier-Stokes equation in a bounded domain},
\newblock \emph{SIAM J. Math. Anal.}, \textbf{41} (2009), 1206-1230.

\bibitem{Ouhabaz} (MR2124040)
\newblock E.~M.~Ouhabaz,
\newblock \emph{Analysis of Heat Equations on Domains,}
\newblock Volume 31 of London Mathematical Society Monographs Series. Princeton University Press, Princeton, 2005.

\bibitem{Pardoux} (MR553909) [10.1080/17442507908833142]
\newblock E.~Pardoux,
\newblock \doititle{Stochastic partial differential equations and filtering of diffusion processes},
\newblock \emph{Stochastics}, \textbf{3} (1979), 127--167.

\bibitem{Pardoux2} (MR0651582)
\newblock E. Pardoux,
\newblock \emph{Equations aux D\'eriv\'ees Partielles Stochastiques Monotones},
\newblock The\`ese de Doctorat, Universit\'e Paris-Sud, 1975.

\bibitem{FS+MSanM}
\newblock {F. Sagu\'es and M. San Miguel,}
\newblock {Dynamics of Fr\'eedericksz transition in a fluctuating magnetic field,}
\newblock \emph{Phys. Rev. A.}, \textbf{32} (1985), 1843--1851.

\bibitem{San Miguel-1985}
\newblock {M.~San Miguel,}
\newblock Nematic liquid crystals in a stochastic magnetic field: Spatial correlations,
\newblock \emph{Phys. Rev. A}, \textbf{32} (1985), 3811--3813.

\bibitem{Shkoller} (MR1916558) [10.1081/PDE-120004895]
\newblock S. Shkoller,
\newblock \doititle{Well-posedness and global attractors for liquid crystal on Riemannian manifolds,}
\newblock \emph{Communication in Partial Differential Equations,} \textbf{27} (2002), 1103--1137.

\bibitem{Temam_1983} (MR764933)
\newblock R. Temam,
\newblock \emph{Navier-Stokes Equations and Nonlinear Functional Analysis},
\newblock CBMS-NSF Regional Conference Series in Applied Mathematics, \textbf{41}, Society for Industrial and Applied Mathematics (SIAM), Philadelphia, PA, 1983.

\bibitem{Wang2014} (MR3268876) [10.1007/s00526-013-0700-y]
\newblock M.~Wang and W.~Wang,
\newblock \doititle{Global existence of weak solution for the 2-D Ericksen-Leslie system,}
\newblock \emph{Calc. Var. Partial Differential Equations,} \textbf{51} (2014), 915--962.

\bibitem{WZZ} (MR3116006) [10.1007/s00205-013-0659-z]
\newblock W.~Wang, P.~Zhang and Z.~Zhang,
\newblock \doititle{Well-posedness of the Ericksen-Leslie system,}
\newblock \emph{Arch. Ration. Mech. Anal.,} \textbf{210} (2013), 837--855.

\bibitem{Wang2016} (MR3503050) [10.3934/dcdsb.2016.21.919]
\newblock M.~Wang, W.~Wang and Z.~Zhang,
\newblock \doititle{On the uniqueness of weak solution for the 2-D Ericksen-Leslie system},
\newblock \emph{Discrete Contin. Dyn. Syst. Ser. B}, \textbf{21} (2016), 919--941.

\bibitem{WZZ15} (MR3366748) [10.1002/cpa.21549]
\newblock W.~Wang, P.~Zhang and Z.~Zhang,
\newblock \doititle{The small Deborah number limit of the Doi-Onsager equation to the Ericksen-Leslie equation},
\newblock \emph{Comm. Pure Appl. Math.}, \textbf{68} (2015), 1326--1398.
	
\end{thebibliography}
\end{document}